\documentclass[11pt,a4paper,reqno]{amsart}
\usepackage{amsfonts}
\usepackage{amsmath,amssymb,amsthm,amsxtra,mathrsfs}
\usepackage{float}
\usepackage{bbm}
\usepackage[colorlinks, linkcolor=blue, anchorcolor=Periwinkle, citecolor=blue, urlcolor=Emerald]{hyperref}
\usepackage[dvipsnames]{xcolor}
\usepackage{geometry,array} 
\usepackage{tikz}\usetikzlibrary{matrix}
\usepackage{url}
\usepackage{dsfont}
\usepackage{tablists} \restorelistitem
\usepackage{extarrows}
\usepackage[all]{xy}
\usepackage{setspace}
\usepackage{hyperref}
\usepackage{tikz-cd}
\usepackage{amscd}
\usepackage{multirow}
\usepackage{longtable}
\usepackage{tabularx}
\usepackage{microtype}
\usepackage{stackengine}

 \usepackage{listings}
 \usepackage{xcolor}
\newcommand{\rank}{\operatorname{rank}}
\newcommand{\im}{\operatorname{im}}

\usepackage{threeparttable}
\usepackage{caption}
\usepackage{adjustbox}

\usetikzlibrary{decorations.markings}
\tikzset{->-/.style={decoration={  markings,  mark=at position #1 with
			{\arrow{>}}},postaction={decorate}}}
\tikzset{-<-/.style={decoration={  markings,  mark=at position #1 with
			{\arrow{<}}},postaction={decorate}}}

\newcommand{\id}{\operatorname{id}}

\newcommand{\Hom}{\operatorname{Hom}}

\renewcommand{\dim}{\operatorname{dim}}

\def\deg{\operatorname {deg}}
\def\GrM{\operatorname {GrMod}}
\def\grmod{\operatorname {grmod}}
\def\Aut{\operatorname{Aut}}
\def\gld{\operatorname{gldim}}

\def\tors{\operatorname{tors}}

\def\qgr{\operatorname{qgr}}
\def\Ext{\operatorname {Ext}}

{\tiny }
\def\diag{\operatorname{diag}}
\def\qgr{\operatorname{qgr}}
\def\supp{\operatorname{supp}}

\newcommand{\mcZ}{\mathcal{Z}}

\newcommand{\mbN}{\mathbb{N}}

\newcommand{\mbP}{\mathbb{P}}

\newcommand{\mbZ}{\mathbb{Z}}

\theoremstyle{plain}
\newtheorem{theorem}{Theorem}[section]

\newtheorem{lemma}[theorem]{Lemma}
\newtheorem{corollary}[theorem]{Corollary}
\newtheorem{proposition}[theorem]{Proposition}

\theoremstyle{definition}
\newtheorem{definition}[theorem]{Definition}

\newtheorem{example}[theorem]{Example}

\newtheorem{remark}[theorem]{Remark}

\newtheorem{question}[theorem]{Question}

\numberwithin{equation}{section}
\newtheorem{definition-proposition}[theorem]{Definition-Proposition}

\begin{document}

\title{Point varieties and point-exactness of Koszul algebras}

\author{Haigang Hu}
\address{Hu: School of Mathematical Sciences, University of Science and Technology of China, Hefei, Anhui 230026, China}
\email{huhaigang@ustc.edu.cn}
\author{Wenchao Wu}
\address{Wu: School of Mathematical Sciences, University of Science and Technology of China, Hefei, Anhui 230026, China}
\email{wuwch20@mail.ustc.edu.cn}
\author{Yu Ye}
\address{Ye: School of Mathematical Sciences, University of Science and Technology of China, Hefei, Anhui 230026, China  
\newline
and 
\newline
Hefei National Laboratory, University of Science and Technology of China, Hefei 230088, China}
\email{yeyu@ustc.edu.cn}

\keywords{Koszul algebra, point variety, point-exact condition}

\thanks{{\it 2020 MSC}: 16E65, 16S38, 16W50}

\begin{abstract}
In this paper, we introduce the point-exact condition for a Koszul algebra $A$, which is useful for characterizing the (G1) condition of $A$ in the sense of Mori. Let $B = A/(f)$, where $f \in A_2$ is a regular normal element. We show that if $A$ satisfies the (G1) condition and is point-exact up to degree $\ell \geq 2$, then $B$ also satisfies the (G1) condition and is point-exact up to degree $\ell$. Moreover, we show that skew polynomial algebras satisfy the point-exact condition.
\end{abstract}

\maketitle

\tableofcontents

\section{Introduction}

In noncommutative algebraic geometry, the study and classification of quantum polynomial algebras and noncommutative quadric hypersurfaces are two major projects.
In the foundational work of Artin, Tate, and Van den Bergh \cite{ATV90}, they introduce {\it point varieties} for connected graded algebras, which are important geometric invariants for $3$-dimensional quantum polynomial algebras $A$. They prove that $A$ is {\it nondegenerate} (or equivalently $A$ satisfies the (G1) condition; see Definition \ref{defn-g1}), meaning that its point variety coincides with the graph of an automorphism of a projective variety $E \subset \mathbb{P}^2$, and $E$ is either the projective plane $\mathbb{P}^2$ or a cubic curve in $\mathbb{P}^2$. This result makes the classification of $3$-dimensional quantum polynomial algebras remarkably elegant and has greatly stimulated the development of this research field. In particular, the study of point varieties of higher-dimensional quantum polynomials remains an active area of research \cite{BDL,HM24,SM}.

Let $A$ be an $n$-dimensional quantum polynomial algebra, and $f_1, \dots, f_r \in A_2$ a regular normal sequence.
The quotient algebra $B = A/(f_1, \dots, f_r)$ is said to be a {\it noncommutative quadratic complete intersection (embedded into a quantum $\mathbb{P}^{n-1}$)}.  In particular, when $r = 1$, $B$ is called a noncommutative quadric hypersurface.  
Recent works show that the point varieties and the (G1) condition are also very useful in the study of Calabi-Yau conics (i.e., $A$ is a $3$-dimensional Calabi-Yau quantum polynomial algebra and $r = 1$) \cite{HMM23}, and Clifford quadratic complete intersections (i.e., $A$ is a Clifford quantum polynomial algebra) \cite{HM24}. 

Note that all the algebras mentioned above are Koszul algebras. Therefore, in this paper we aim to answer a very natural question.

\begin{question} \label{que-intr}
Let $A$ be a Koszul algebra, $f \in A_2$ a regular normal element, and $B = A/(f)$. 
If $A$ satisfies the (G1) condition, does it imply that $B$ also satisfies the (G1) condition?
\end{question}

To address this question, we introduce the concept of point-exactness for Koszul algebras $A$ (see Definition \ref{def-pe}). Following \cite{ATV90}, we show that $A$ satisfies the (G1) condition if and only if $A$ is semi‑standard and point‑exact at degree $1$ (see Lemma \ref{lem-E1}). The crucial step is then the explicit construction of a linear resolution of $\Bbbk_B$ from a linear resolution of $\Bbbk_A$. In fact, we consider a more general setting in Section \ref{sec-resquo}, as stated in the following theorem.

\begin{theorem}[Theorem \ref{T1}] \label{thm.introt1}
Let $A$ be a connected graded algebra, $f \in A_m$ a regular normal element with $m \geq 1$, and $B = A/(f)$. Let $P_\bullet \xrightarrow{\sim} \Bbbk_A$ be a minimal resolution.
The sequence $T_\bullet$ defined in (\ref{miniB}), which is constructed from $P_\bullet$, is a free resolution of $\Bbbk_B$.
If in addition $A$ is Koszul and $m \geq 2$, then $T_\bullet \xrightarrow{\sim} \Bbbk_B$ is a minimal  resolution. 
\end{theorem}

Let $A$ and $B$ be as in Question \ref{que-intr}. Assuming that $A$ is point‑exact up to degree $\ell \geq 2$ and using the linear resolution of $\Bbbk_B$ from Theorem \ref{thm.introt1}, we answer Question \ref{que-intr} affirmatively.

\begin{theorem}[Theorem \ref{thm-main}] \label{thm.introt2}
Keep notations as in Question \ref{que-intr}.
If $A$ satisfies the (G1) condition and is point‑exact up to degree $\ell \geq 2$, then $B$ also satisfies the (G1) condition and is point‑exact up to degree $\ell$.
\end{theorem}

Furthermore, we show that the following algebras satisfy the point‑exact condition (i.e., they are point‑exact at every degree $\geq 1$):
\begin{itemize}
\item[(1)] $2$ and $3$-dimensional quantum polynomial algebras (Proposition \ref{prop.34qec}); 
\item[(2)] $4$-dimensional quantum polynomial algebras satisfying the (G1) condition (Proposition \ref{prop.34qec}); 
\item[(3)] skew polynomial algebras of dimension $\geq 2$ (Definition \ref{defn-spa}, Theorem \ref{thm-spa}).
\end{itemize}
Thus by Theorem \ref{thm.introt2} and Remark \ref{rem-asg}, if $B$ is a noncommutative quadratic complete intersection embedded into a quantum $\mathbb{P}^{n-1}$ that is defined by one of the above algebras, then $B$ satisfies the (G1) condition and point-exact condition. In particular, the polynomial algebra $\Bbbk[x_1, \dots, x_n]$ is a special case of skew polynomial algebras of dimension $n$, and hence quadratic complete intersections $\Bbbk[x_1,\dots,x_n]/(f_1, \dots, f_r)$ satisfy the point-exact condition. 

\subsection*{Acknowledgement.}
H. Hu was supported by the National Natural Science Foundation of China (Nos.\,12501054 and 12371042). 
Y. Ye was partially supported by the National Key R\&D Program of China
(No.\,2024YFA1013802), the National Natural Science Foundation of China (Nos.\,12131015 and 12371042), and the Quantum Science and Technology-National Science and Technology Major Project (No.\,2021ZD0302902).

\section{Preliminaries}\label{sec2}
Throughout this paper, we work over an algebraically closed field $\Bbbk$ of characteristic $0$. All algebras and vector spaces considered are over $\Bbbk$ and all graded algebras are $\mathbb{Z}$-graded.   

Let $A$ be a  graded algebra. We denote by $\GrM A$ the category of right graded  $A$-modules with degree preserving right $A$-module homomorphisms, and by $\grmod A$ the full subcategory consisting of finitely generated modules. Morphisms in $\GrM A$ are composed from right to left, that is the composite of $f\colon  X \to Y$ and $g\colon Y \to Z$ is denoted by $gf = g\circ f\colon X \to Z$. For $M\in \GrM A$ and $n\in \mbZ$, we define  the shift  $M(n)$ by the new grading $M(n) = M_{n+i}$. For $M,N\in \GrM A$ and $i\in \mbN$, we define
 \[ 
 \underline{\Ext}_A^i(M,N) := \bigoplus_{n\in \mbZ}{\Ext}^i_{\GrM A}(M,N(n)). 
 \]  
 
For a graded algebra $A = \bigoplus_{i \in \mathbb{Z}} A_{i}$, we say that $A$ is \textit{connected graded} if  $A_{0} = \Bbbk$ and $A_{i} = 0$ for all $i < 0$, and $A$ is \textit{locally finite} if $\dim_{\Bbbk} A_{i} < \infty$ for all $i \in \mathbb{Z}$. 
If $A$ is a locally finite graded algebra and $M \in \mathrm{grmod}\,A$, we define the \textit{Hilbert series} of $M$ by $H_{M}(t) := \sum_{i \in \mathbb{Z}} (\dim_{\Bbbk} M_{i})t^{i} \in \mathbb{Z}[[t, t^{-1}]]$.

Let $\sigma$ be a graded automorphism of a connected graded algebra $A$. 
Then $\sigma$ induces an autoequivalence of $\GrM A$: For $M \in \GrM A$, $M^\sigma$ equals  $M$ as vector spaces and is endowed with a new right $A$-action $*$ defined by  $m * a = m \sigma(a)$, for all $m \in M^\sigma$, $a \in A$. 
For a morphism $\varphi: M \to N$ in $\GrM A$. The morphism $\varphi^\sigma: M^\sigma \to N^\sigma$ is defined by $\varphi^\sigma(m) = \varphi(m)$ for all $m \in M^\sigma$. 
The equivalence $(-)^\sigma$ commutes with the grading shift functor. 

Let $A$ be a connected graded algebra, and $F_0, F_1$ free objects in $\GrM A$. We often view an element $x \in F_0$ as a column vector over $A$, and a morphism 
$\varphi: F_0 \to F_1$ in $\GrM A$ as a matrix over $A$, i.e., $\varphi = (\varphi_{ij}) \in \mathbb{M}_{s_1 \times s_0}(A)$, where $\rank_A F_i = s_i$ for $i = 0,1$. 
Moreover, for a graded automorphism $\sigma$ of $A$, we define $\sigma \cdot \varphi : = (\sigma (\varphi_{ij})) \in \mathbb{M}_{s_1 \times s_0}(A)$.

\begin{definition}Let $ A $ be a graded algebra and $ f \in A_m $.
	\begin{enumerate}
		\item[(1)] We say that $ f $ is \emph{regular} if, for every $ a \in A $, $ af = 0 $ or $ fa = 0 $ implies $ a = 0 $.
		\item[(2)] We say that $ f $ is \emph{normal} if $ Af = fA $.
	\end{enumerate}
\end{definition}

We remark that $f$ is regular and normal if and only if there exists a unique (graded) ring automorphism $ \sigma$ of $A$ such that $ fa = \sigma(a)f $ for $ a \in A$. We call $ \sigma$ the \emph{normalizing automorphism} of $ f $.

\subsection{Matrices over a vector space}
Let $V$ be an $n$-dimensional vector space with a basis $\{ x_1, \dots, x_n\}$. 
Let $M = (m_{ij}) \in \mathbb{M}_{s \times r} (V)$. We regard every minor of $M$ as a (homogeneous) polynomial in $\Bbbk[x_1,\dots,x_n]$.

Let $t \in \mathbb{N}$, we define $X^t_M$ to be the projective variety
$$
X^t_M : = \mathcal{Z} (\{t\text{-minors of } M\}) \subset \mathbb{P}(V^*). 
$$
Note that $X^t_M =  \mathbb{P}(V^*)$ if $t > \min\{s,r\}$. For convenience, we will omit the superscript $t$ in $X^t_M$ if $t = n$. 

For $p = (p_1 \mathpunct{:} \, \cdots \, \mathpunct{:} p_n) \in \mathbb{P}(V^*)$, define $M_p$ to be the matrix over $\Bbbk$ by replacing $x_i$ by $p_i$ for $i = 1, \dots,n$. Then $M_p$ is well-defined up to scalar. 
We are also able to define $\rank_\Bbbk M_p$ in the usual sense,
and clearly, $\rank_\Bbbk M_p < t$ if and only if $p \in X^t_M$.  

The following is an easy application of rank–nullity theorem. 

\begin{lemma} \label{lem-rk}
Let $M \in \mathbb{M}_{s \times r} (V)$, $N \in \mathbb{M}_{r \times l} (V)$, and $p, q \in \mathbb{P}(V^*)$ such that $M_p N_q = 0$. 
\begin{itemize}
\item[(1)] If $\rank_\Bbbk M_p \geq t$ for some $t \leq \min \{s, r\}$, then $\rank_\Bbbk N_{q} \leq r - t$, so that $q \in X^{r-t+1}_{N}$. 
\item[(2)] If $\rank_\Bbbk N_q \geq t$ for some $t \leq \min \{r, l\}$, then $\rank_\Bbbk M_{p} \leq r - t$, so that $p \in X^{r-t+1}_{M}$. 
\end{itemize}
\end{lemma}

\subsection{Koszul algebras}
Let $V$ be a finite dimensional vector space, and $T(V)$ the tensor algebra. 

\begin{definition}
A quotient algebra $A = T(V)/(R)$ is called a {\it quadratic algebra} if elements in $V$ have degree $1$ and $R \subset V^{\otimes 2}$.
\end{definition}

\begin{definition} 
A connected graded algebra $A$ is called {\it Koszul} if the canonical module $\Bbbk_A = A/A_{\geq 1}$ admits a linear resolution 
$P_{\bullet} = (P_i, d_i)$. Which is to say, there is an exact sequence 
$$
\xymatrix{\cdots \to P_2 \ar[r]^-{d_2} &  P_1 \ar[r]^-{d_1} & P_0 \to \Bbbk_A \to 0}
$$ 
where $P_i$ is projective and generated by its component in degree $i$, i.e., $P_i = (P_{i})_iA$. 
\end{definition}

\begin{proposition} \label{prop-kos}
Let $A$ be a Koszul algebra. Then the following statements hold. 
\begin{itemize}
\item[(1)] $A$ is quadratic \cite[Corollary 2.3.3]{BGS}.
\item[(2)] The opposite algebra $A^{op}$ is Koszul \cite[Proposition 2.2.1]{BGS}.
\item[(3)] Let $f\in A_2$ be a regular normal element. Then $A/(f)$ is Koszul \cite[Theorem 1.2]{ST01}. 
\end{itemize}
\end{proposition}

\begin{definition}[\cite{AS87}]
{\rm
A  noetherian connected graded algebra $A$ is called a {\it quantum polynomial algebra of dimension $n$} if
\begin{itemize}
\item[(1)]  $\gld  A = n < \infty$,
\item[(2)] $\underline{\Ext}^i_A(\Bbbk,A)\cong \underline{\Ext}^i_{A^{op}}(\Bbbk,A) \cong  \begin{cases}0 & i \neq n \\ \Bbbk(n) & i = n\end{cases}$, 
\item[(3)] $H_A(t)=(1-t)^{-n}$.
\end{itemize}
}
\end{definition}

 Zhang shows that every quantum polynomial algebra is Koszul \cite[Theorem 5.11]{Shi96}.

\subsection{Point varieties}
Let $A = T(V)/I$ be a connected graded algebra finitely generated in degree $1$. 
Then each element in $I_i \subset V^{\otimes i}$ defines a multilinear form 
$
(V^*)^{\times i} \to \Bbbk. 
$
Define a projective scheme associated to $I_i$ by
$$
\mathcal{V}(I_i) 
:= \{ (p_1, \dots, p_{i}) \in \mathbb{P}(V^*)^{\times i} \mid f (p_1, \dots, p_{i}) = 0, \; \forall f \in I_i \}.
$$
Denote $\Gamma_i = \mathcal{V}(I_i)$. For $j \geq i$, let $p^j_i: \Gamma_j \to \Gamma_i$ be the restriction of the projection $\mathbb{P}(V^*)^{\times j} \to \mathbb{P}(V^*)^{\times i}$ onto the first $i$ factors. Then $\{\Gamma_i, p^j_i \}$ becomes an inverse system of schemes.

\begin{definition}[\cite{ATV90}]
The {\it point scheme} of $A$ is defined to be the inverse limit $\Gamma: = \varprojlim \Gamma_i$. If we consider only the reduced structure of $\Gamma$, we call it the {\it point variety}. 
\end{definition}

In this paper, we mainly consider quadratic algebras. The following definition is useful for characterizing the point varieties of quadratic algebras.

\begin{definition}[\cite{Mo06}] \label{defn-g1}
We say a quadratic algebra $A = T(V)/(R)$ satisfies the  {\it (G1) condition} if there is a pair $(E, \sigma)$, where $E\subset \mathbb{P}(V^*)$ is a projective variety and $\sigma\in \Aut_\Bbbk E$, such that
\[
\mathcal{V}(R)=\{(p, \sigma(p))\in \mathbb{P}(V^*)\times \mathbb{P}(V^*)\mid p\in E\}.
\]
If $A$ satisfies the (G1) condition, the pair $(E,\sigma)$ is called the {\it geometric pair} of $A$.
\end{definition}

We mention that a commutative quadratic algebra $A = \Bbbk[x_1, \dots, x_n]/I$ always satisfies the (G1) condition, with geometric pair $(E = \mathcal{Z}(I), \sigma = \id)$ \cite[Lemma 2.5]{HMM23}. 

If a quadratic algebra $A$ satisfies the (G1) condition with geometric pair $(E,\sigma)$, then its point variety $\Gamma \cong E$ by \cite{ATV90} (see also \cite[Lemma 4.1]{V10}).

\begin{remark}\label{rem-asg}
The following algebras satisfy the (G1) condition:
\begin{itemize}
\item[(1)] $2$ and $3$-dimensional quantum polynomial algebras \cite[Theorem 1]{ATV90};
\item[(2)] Calabi-Yau conics \cite[Proposition 4.5]{HMM23};
\item[(3)] Auslander-regular 4-dimensional quantum polynomial algebras \cite[Theorem 1.4]{SV99};
\item[(4)] skew polynomial algebras (of dimension $\geq 2$) \cite[Proposition 4.2]{V10} (see also \cite{BDL});
\item[(5)] Clifford quantum polynomial algebras (of dimension $\geq 2$) and Clifford quadratic complete intersections \cite[Theorem 1.4]{HM24}. 
\end{itemize} 
\end{remark}

\section{Point-exact conditions} \label{sec3}

In this section, $A = T(V)/(R) = \Bbbk \langle x_1,\dots,x_n \rangle/(f_1,\dots, f_r)$ is a Koszul algebra.
We may write 
$$
f_i = \sum_{j=1}^{n}x_jf_{ji}=\sum_{j=1}^{n}g_{ij}x_j
$$
for every $1\leq i\leq r$. Then $\Bbbk_A$ has a linear resolution
\begin{align}
P_\bullet: \ \ \ \cdots \to A(-i)^{\oplus s_i} \xlongrightarrow{d_i}  \cdots \to A(-2)^{\oplus s_2} \xlongrightarrow{d_2}  A(-1)^{\oplus s_1} \xlongrightarrow{d_1} A \to  0 \label{lr1}
\end{align}
where $s_1 = n$, $s_2 = r$, $d_1=(x_1,\dots,x_n)$, and $d_2=(f_{ij}) \in \mathbb{M}_{n\times r}(V)$. Similarly, by Proposition \ref{prop-kos} (2),  ${_A\Bbbk}$ has a linear resolution
\begin{align}
Q_\bullet:  \  \  \ \cdots \to A(-i)^{\oplus s_i} \xlongrightarrow{h_i} \cdots \to A(-2)^{\oplus s_2} \xlongrightarrow{h_2}  A(-1)^{\oplus s_1} \xlongrightarrow{h_1} A \to  0 \label{lr2}
\end{align}
where $ h_1 = (x_1,\dots,x_n)^t$ and $h_2=(g_{ij}) \in \mathbb{M}_{r \times n}(V)$. 
We define
$$
X_A := X_{d_2} = \mathcal{Z} (\{n\text{-minors of } d_2\}),\; 
{_A}X := X_{h_2} = \mathcal{Z} (\{n\text{-minors of } h_2\}) \subset \mathbb{P}^{n-1}.
$$
 We mention that $X_A = {_A}X = \mathbb{P}^{n-1}$ if $n > r$.

\begin{remark}\label{zero-wd}
 	Let $P_{\bullet}'$ be another linear resolution of $\Bbbk_A$. By the uniqueness of linear resolutions up to isomorphism, there exists a chain isomorphism $\varphi=(\varphi_i)_{i\geq 0}: P_{\bullet} \rightarrow P_{\bullet}'$ with $\varphi_i \in \mathbb{M}_{s_i}(\Bbbk)$ for all $i\geq 0$. Thus  $X_A$ is well-defined for $A$. Similarly, ${_A}X$ is also well-defined for $A$.
\end{remark} 

The following result follows from the constructions of $P_\bullet$ and $Q_\bullet$. 

\begin{lemma} \label{lem-vrdh}
We have  
\begin{align*}
\mathcal{V}(R) & = \{(p,q) \in \mathbb{P}(V^*) \times \mathbb{P}(V^*) \mid (d_1)_p (d_2)_q = 0\}    \\
& = \{(p,q) \in \mathbb{P}(V^*) \times \mathbb{P}(V^*) \mid (h_2)_p (h_1)_q = 0\}.
\end{align*} 
Which implies that the following are equivalent.
\begin{itemize}
\item[(1)] $\mathcal{V}(R) = \emptyset$.
\item[(2)] $X_A = \emptyset$.
\item[(3)] ${_A}X = \emptyset$.
\end{itemize}
\end{lemma}

\begin{remark}
In this paper, we are only interested in the case where $\mathcal{V}(R) \neq \emptyset$. 
Thus we will always assume that $X_A$ and ${_A}X$ are non-empty.
\end{remark}

\begin{definition} \label{def.sest}
We say $A$ is {\it semi-standard} if $X_A={_A}X$. 
\end{definition}

\begin{definition}\label{def-pe}
\begin{itemize}
\item[(1)] For $i \geq 1$, we say $A$ is {\it right point-exact at degree $i$} if 
\begin{equation}\label{rpe}
	\rank_\Bbbk (d_{i})_p + \rank_\Bbbk (d_{i+1})_p = \rank_A P_i = s_i, \ \forall p \in X_A.
\end{equation}
\item[(2)] For $\ell \geq 1$, we say $A$ is {\it right point-exact up to degree} $\ell$ if it is right point exact at degree $i$ for all $1 \leq i \leq \ell$. 
\item[(3)] We say $A$ satisfies the {\it right point-exact condition} if it is right point-exact at every degree $\geq 1$. 
\end{itemize}
\end{definition}

The left version of the above definition is analogous. We say $A$ is {\it point-exact at degree $i$} (resp. {\it up to degree $\ell$}) if it is right and left point-exact at degree $i$ (resp. up to degree $\ell$). Moreover, $A$ satisfies the {\it point‑exact condition} if it satisfies both the right and left point‑exact conditions. 

\begin{remark}
	The right (left) point-exact condition is well-defined by the same reasoning as in Remark $\ref{zero-wd}$.
\end{remark}

\begin{lemma} \label{lem-rkdi}
The following are equivalent.
\begin{itemize}
\item[(1)] $A$ is right point-exact up to degree $ \ell \geq 1$. 
\item[(2)] $\rank_\Bbbk (d_i)_p = \sum_{j=1}^i (-1)^{j+1} \rank_\Bbbk P_{i-j}, \ \forall p \in X_A$, where $1 \leq i \leq \ell + 1$.
\end{itemize}
\end{lemma}

We attribute the proof of the following lemma to \cite{ATV90}.

\begin{lemma}\label{lem-E1}
The following are equivalent.
\begin{itemize}
\item[(1)] $A$ satisfies the (G1) condition.
\item[(2)] $A$ is semi-standard and point-exact at degree $1$. 
\end{itemize}
Moreover, if $A$ satisfies the above equivalent conditions, then $r+1 \geq n$ and the geometric pair $(E, \sigma)$ associated to $A$ is given by $E = X_A = {_A}X$, and 
$
\sigma: E \to E, \; p \mapsto q 
$
such that $(d_1)_p (d_2)_q = 0$ (or equivalently $(h_2)_p (h_1)_q= 0$).
\end{lemma}

\begin{proof}
Assume $A$ satisfies the (G1) condition with geometric pair $(E,\sigma)$. Let $q \in X_A = X_{d_2}$. Then there is $p \in \mathbb{P}^{n-1}$ such that
$$
(d_1)_p (d_2)_q = (f_1 (p,q), \dots, f_r (p,q)) = 0
$$
since $\rank_\Bbbk (d_2)_q < n$. Hence $(p,q) \in \mathcal{V}(R)$, which implies that $p, q\in E$ and $\sigma(p) = q$ by assumption. It follows that $p$ is uniquely determined by $q$. Thus $\rank_\Bbbk(d_2)_q = n-1$, i.e., $A$ is right point-exact at degree $1$. 
A similar argument shows that $A$ is left point‑exact at degree $1$.
On the other hand, for any $(p,q)\in \mathcal{V}(R)$, we have 
$$
(h_2)_p (h_1)_q  = (f_1 (p,q), \dots, f_r (p,q))^t = 0. 
$$
Since $\rank_\Bbbk (h_1)_q = 1$, Lemma \ref{lem-rk} (2) gives $p \in X_{h_2} = {_A}X$. 
Therefore $X_A \subset E \subset {_A}X$. The reverse inclusion ${_A}X \subset E \subset X_A$ is proved similarly.

Now assume that $A$ is semi-standard and point-exact at degree $1$.
Let  $\pi_1, \pi_2: \mbP(V^*) \times \mbP(V^*) \to \mbP(V^*)$ be the canonical projections. 
Then the restriction
$$
\pi_1|_{\mathcal{V}(R)} : \mathcal{V}(R) \to {_A}X =: E, \, (p,q) \mapsto p
$$
is a well-defined morphism of varieties. Since for any $p \in E$, $(h_2)_p (h_1)_q = 0$ implies that $q$ is uniquely determined by $p$, $\pi_1|_{\mathcal{V}(R)}$ is an isomorphism. Similarly, $\pi_2|_{\mathcal{V}(R)}: \mathcal{V}(R) \to X_A = E$ is also an isomorphism. 
Then 
\begin{equation*}
\mathcal{V}(R) = \{ (p, \sigma(p)) \mid p \in E, \, \sigma =  \pi_1|_{\mathcal{V}(R)} \circ (\pi_2|_{\mathcal{V}(R)})^{-1} \in \Aut_\Bbbk E\}. \qedhere
\end{equation*}
\end{proof}

\begin{example} \label{exm-2dimqpa}
Let $A = k\langle x,y\rangle/(xy-\lambda yx), 0\neq \lambda \in \Bbbk$. Then $A$ is a $2$-dimensional quantum polynomial algebra, and we have linear resolutions
$$
\xymatrix{0 \to  A(-2) \ar[r]^-{d_2} & A(-1)^{\oplus 2} \ar[r]^-{d_1} & A \to \Bbbk_A \to 0}
$$ 
where $d_1 = (x, y)$, $d_2 = (y, - \lambda x)^t$, and
$$
\xymatrix{0 \to  A(-2) \ar[r]^-{h_2} & A(-1)^{\oplus 2} \ar[r]^-{h_1} & A \to {_A\Bbbk} \to 0}
$$ 
where $h_1 = (x, y)^t$, $d_2 = (-\lambda y, x)$. 
Then $X_A = {_A}X =\mbP^1$, and 
$$
\rank_\Bbbk (d_1)_p + \rank_\Bbbk (d_2)_p = 2  = \rank_\Bbbk (h_1)_p + \rank_\Bbbk (h_2)_p, \ \forall p \in \mbP^1. 
$$
Thus $A$ is semi-standard and point exact at degree $1$. By Lemma \ref{lem-E1}, $A$ satisfies the (G1) condition with geometric pair $(E, \sigma)$ where $E = \mbP^1$, and $\sigma: \mbP^1 \to \mbP^1, \; p = (p_1 \mathpunct{:} p_2) \mapsto \sigma(p) = ( p_1 \mathpunct{:} \lambda p_2)$.
\end{example}

The following is an interesting observation.

\begin{lemma}\label{lem-eqrep}
Assume $A$ satisfy the (G1) condition with geometric pair $(E,\sigma)$. Then $A$ satisfies the right point-exact condition if and only if  for every $p \in E$ the  complex 
	\begin{align}
		\cdots \to \Bbbk^{s_n} \xlongrightarrow{(d_n)_{\sigma^{n-1}(p)}} \cdots \to \Bbbk^{s_{2}}  \xlongrightarrow{(d_2)_{\sigma(p)}} \Bbbk^{s_1} \xlongrightarrow{(d_1)_p} \Bbbk \rightarrow 0  \label{def-rep}
	\end{align}
	is exact.
\end{lemma}

\begin{proof}
The complex in (\ref{def-rep}) is exact for every $p\in E$ if and only if 
$$
\rank_\Bbbk (d_i)_p = \sum_{j=1}^i (-1)^{j+1} s_{i-j}, \, \forall i \geq 1, \forall p \in E.
$$
The result follows from Lemma \ref{lem-rkdi}.
\end{proof}

\begin{proposition}\label{prop.34qec}
The following algebras satisfy the point-exact condition. 
\begin{itemize}
\item[(1)] $2$ and $3$-dimensional quantum polynomial algebras.
\item[(2)] $4$-dimensional quantum polynomial algebras satisfying the (G1) condition. 
\end{itemize}
\end{proposition}

\begin{proof}
We prove (2). Let $A$ be a $4$-dimensional quantum polynomial algebra that satisfies the (G1) condition. Then $\Bbbk_A$ admits a linear resolution
\[ 0\rightarrow A(-4) \xrightarrow{d_4}A(-3)^{\oplus 4} \xrightarrow{d_3} A(-2)^{\oplus 6} \xrightarrow{d_2} A(-1)^{\oplus 4} \xrightarrow{d_1} A \to \Bbbk_A \to  0,\] 
where $d_1=(x_1,\dots,x_n)$ and $d_4=(x_1,\dots,x_n)^t$.
By applying $\underline{\Hom}(-,A)$ and the grading shift functor, we obtain a linear resolution of $_A\Bbbk$:
\[ 0\rightarrow A(-4) \xrightarrow{d_1}A(-3)^{\oplus 4} \xrightarrow{d_2} A(-2)^{\oplus 6} \xrightarrow{d_3} A(-1)^{\oplus 4} \xrightarrow{d_4} A \rightarrow {_A \Bbbk} \to 0 .\] 
Since $A$ satisfies the (G1) condition, $X_{d_2} = X_A = {_A}X =  X_{d_3} $ and $\rank_\Bbbk (d_{2})_p = \rank_\Bbbk (d_{3})_p =3 $ for every $p\in X_A$ by Lemma \ref{lem-E1}. Thus $A$ is point-exact by Lemma \ref{lem-rkdi}.
\end{proof}

\section{Minimal resolutions}\label{sec-resquo}

In this section, $A$ is a connected graded algebra, $f \in A_m$ is a regular normal element with $m \geq 1$, and $B = A/(f)$. 
Let $P_\bullet = (P_i, d_i)\xrightarrow{\varepsilon}\Bbbk_A \to 0$ be a minimal resolution. 
We will construct a free resolution $T_\bullet\xrightarrow{\sim}\Bbbk_B$ from $P_\bullet$ by generalizing Shamash's method in the commutative case \cite{Sh69}. 

Let $\sigma$ be the normalizing automorphism of $f$ and set $\tau := \sigma^{-1}$. 
The normalizing automorphism induces an isomorphism 
$
\tilde{\sigma}: A \rightarrow A^{\sigma}, \, a \mapsto \sigma(a)
$
in $\GrM A$. 
We also define 
$\gamma: A^{\sigma}(-m)\rightarrow  A$, $a \mapsto af$, which is a morphism
in $\GrM A$. 

Let $F_0$, $F_1$ be free objects in $\GrM A$, and $\varphi: F_0 \to F_1$ a morphism in $\GrM A$. Recall that we define $\tau \cdot \varphi = (\tau(\varphi_{ij}))$ by regarding $\varphi = (\varphi_{ij})$ as a matrix in $\mathbb{M}_{\rank_A F_1 \times \rank_A F_0}(A)$. 
Then we have the following commutative diagram
\begin{equation} \label{twist}
\begin{tikzcd}
F_0(-m) \arrow[r, "\tilde{\sigma}_0"] \arrow[d, "\tau \cdot \varphi"']  & F_0^\sigma(-m) \arrow[r, "\gamma_0"] \arrow[d, "\varphi^{\sigma}"'] 
& F_0\arrow[d, "\varphi"]  \\
F_1(-m) \arrow[r, "\tilde{\sigma}_1"]   &
F_1^\sigma(-m) \arrow[r, "\gamma_1"] & 	
F_1,
\end{tikzcd}
\end{equation}
where $\tilde{\sigma}_i$ and $\gamma_i$ are morphisms induced by $\tilde{\sigma}$ and $\gamma$  (coordinate-wisely) for $i = 0,1$. Since $\gamma\circ \sigma$ is left multiplication by $f$, we have 
$$
\gamma_i \circ  \tilde{\sigma}_i  = f I_{\rank_A F_i} \in \mathbb{M}_{\rank_A F_i}(A), 
$$
where $I_{\rank_A F_i}$ is the identity matrix. 

Let $n \in \mbZ$. Denote by $P_\bullet(-n)$ the complex
$$
\cdots \to P_i(-n) \xrightarrow{d_i} \cdots \to P_2(-n) \xrightarrow{d_2} P_1(-n) \xrightarrow{d_1} P_0(-n) \to 0,
$$
and $P_\bullet[-n]$ the {\it shifted complex}: $P_\bullet[-n]_{i} = P_{i+n}$ with differential $(-1)^n d$.
Define a sequence
$$
\tau^n P_\bullet: \ \ \ \cdots \to P_i \xrightarrow{\tau^n \cdot d_i} \cdots \to P_2 \xrightarrow{\tau^n \cdot d_2} P_1 \xrightarrow{\tau^n \cdot d_1} P_0  \to 0
$$
in $\GrM A$. 

Using the commutative diagram in (\ref{twist}), we obtain the following commutative diagram: 
\begin{equation} \label{com.ppp}
\begin{tikzcd}
\cdots \arrow[r] & P_{i+1}(-m) \arrow[r, "\tau \cdot d_{i+1}"] \arrow[d, "\tilde{\sigma}_{i+1}"']  & P_i(-m) \arrow[r, "\tau \cdot d_i"] \arrow[d, "\tilde{\sigma}_{i}"'] & P_{i-1}(-m) \arrow[d, "\tilde{\sigma}_{i-1}"']  \arrow[r] & \cdots\\
\cdots \arrow[r] & P^\sigma_{i+1} (-m) \arrow[r, "(d_{i+1})^\sigma"]   \arrow[d, "\gamma_{i+1}"']& P^\sigma_i(-m) \arrow[r, "(d_{i})^\sigma"] 
\arrow[d, "\gamma_{i}"']&  P^\sigma_{i-1}(-m) \arrow[r] \arrow[d, "\gamma_{i-1}"']& \cdots \\
\cdots \arrow[r] & P_{i+1} \arrow[r, "d_{i+1}"]   & P_i  \arrow[r, "d_i"] &  P_{i-1} \arrow[r] & \cdots.
\end{tikzcd}
\end{equation}
In the diagram above, since the sequence $P_\bullet$ in the bottom row is a minimal resolution of $\Bbbk_A$, the sequence in the middle row, which can be obtained by applying the equivalence $(-)^\sigma$ to $P_\bullet(-m)$, is a minimal resolution of $\Bbbk_A^\sigma (-m) = \Bbbk_A (-m)$. Moreover, since $\tilde{\sigma}_i$ induced by $\tilde{\sigma}$ is an isomorphism for every $i \in \mathbb{Z}$, the sequence $\tau P_\bullet(-m)$ in the top row is a minimal resolution of $\Bbbk_A(-m)$. 
Thus we have proved the following result. 

\begin{lemma}\label{lem-tplin}
The sequence $\tau^n P_\bullet = (P_i, \tau^n \cdot d_i)$ is a minimal resolution of $\Bbbk_A$ for any $n \in \mathbb{Z}$.
\end{lemma}

Let $s_i = \rank_A P_i$, and let $I_{s_i} \in \mathbb{M}_{s_i}(A)$ be the identity matrix. The map
$$
\zeta^1:= (\zeta^1_i)_{i\in \mbZ} = (fI_{s_i})_{i\in \mbZ} = (\gamma_i \circ  \tilde{\sigma}_i )_{i\in \mbZ}: \tau P_\bullet(-m) \rightarrow P_\bullet
$$
of complexes is a chain map over the left multiplication $f \cdot : \Bbbk_A \to \Bbbk_A, \; a \mapsto fa = 0$. 
Which is to say, we have a commutative diagram
$$
\xymatrix{\tau P_\bullet(-m) \ar[d] \ar[r]^-{\zeta^1} & P_\bullet \ar[d] \\\Bbbk_A \ar[r]^-{0} & \Bbbk_A}
$$
where vertical maps are the augmentation maps. 
Thus $\zeta^1$ is homotopic to zero by comparison theorem \cite[Theorem 6.16]{Ro09}.  So there is a 
sequence of morphisms $c^1 = (c^1_i:  P_i(-m) \rightarrow P_{i+1} )_{ i \in \mathbb{Z} }$ in $\GrM A$ such that $\zeta^1_i = c^1_{i-1} (\tau \cdot d_i)+d_{i+1} c^1_i$.  

\begin{lemma} \label{lem.zeta2}
The  sequence of morphisms
$$
\zeta^2 := (\zeta^2_{i})_{i\in \mbZ} = (-c^1_{i+1}  (\tau \cdot c^1_{i}))_{i\in \mbZ}: \tau^2 P_\bullet(-2m) \rightarrow P_\bullet[-2]
$$
is a map of complexes. 
\end{lemma}

\begin{proof}
We have
\begin{align*}
	c^1_{i + 1} (\tau \cdot c^1_i) (\tau^2 \cdot d_{i + 1})
	& = c^1_{i + 1} \left( \zeta^1_{i+1} - (\tau \cdot d_{i + 2}) (\tau \cdot c^1_{i + 1} ) \right) \\
	& = c^1_{i + 1} \zeta^1_{i+1} - \left( \zeta^1_{i+2} - d_{i + 3} c^1_{i + 2} \right) (\tau \cdot c^1_{i + 1}) \\
	& =\left(c^1_{i + 1} \zeta^1_{i+1} - \zeta^1_{i+2} (\tau \cdot c^1_{i + 1})\right) + d_{i + 3} c^1_{i + 2} (\tau \cdot c^1_{i + 1}) \\
	& = d_{i + 3} c^1_{i + 2} (\tau \cdot c^1_{i + 1}).
\end{align*} 
The last equality follows from the commutative diagram (\ref{twist}).
\end{proof}

Consider the map  $\overline{\zeta^2_0}: {\rm Coker}(\tau^2\cdot d_1)\rightarrow {\rm Coker}(d_3)$ induced by the commutative diagram
$$
\xymatrix{
\tau^2 P_\bullet(-2m)_1 \ar[r]^-{\zeta^2_1} \ar[d]_-{\tau^2\cdot d_1} &   P_\bullet[-2]_1 \ar[d]^-{d_3} \\
\tau^2 P_\bullet(-2m)_0 \ar[r]^-{\zeta^2_0} &  P_\bullet[-2]_0. 
}
$$ 
By Lemma \ref{lem.zeta2}, $\overline{\zeta^2_0} $ is a zero map. Thus $(\zeta^2_i)_{i\geq 0}$ is a chain map over $\overline{\zeta^2_0} = 0$, then by comparison theorem, there is a sequence of morphisms 
$
c^2= (c^2_i:  \tau^2 P_\bullet(-2m)_i \rightarrow P_\bullet[-2]_{i+1})_{ i \in \mathbb{Z} }
$
in $\GrM A$, 
with $c^2_i = 0$ for $i<0$, 
such that $\zeta^2_i = c^2_{i - 1} (\tau^2 \cdot d_i)+d_{i + 1} c^2_i$ for all $i \in \mathbb{Z}$.

\begin{lemma}\label{R0}
For any integer $n \geq 1$, there is a chain map  $\zeta^n: \tau^n P_\bullet(-nm) \rightarrow P_\bullet[-2n+2]$, and sequences of morphisms $c^i = (c^i_l: \tau^i P_\bullet(-im)_l \rightarrow P_\bullet[-2i+2]_{l+1})_{l \in \mathbb{Z}}$ in $\GrM A$ for $1 \leq i \leq n$ such that 
$$
\zeta^n = \begin{cases} 
(fI_{s_i})_{i\in \mbZ} = c^1 (\tau \cdot d)+d c^1 & \text{if } n = 1\\
-\sum_{0 < i < n}c^{i} (\tau^i \cdot c^{n - i}) = c^n  (\tau^n \cdot d)  + d c^n & \text{if } n > 1 \end{cases}
$$
where $s_i = \rank_A P_i$. 

\end{lemma}

\begin{proof}
We have discussed the cases $n=1, 2$. We will show the rest by induction. 

Assume chain maps $\zeta^j : \tau^j P_\bullet(-jm) \rightarrow P_\bullet[-2j+2]$, 
and sequences of morphisms 
$$
c^i = (c^i_l: \tau^i P_\bullet(-im)_l \rightarrow P_\bullet[-2i+2]_{l+1})_{l \in \mathbb{Z}}
$$ 
have been defined for all $1< j < n$, $ 1 \leq i < n$, where $n>2$,
such that 
$$
\zeta^j=-\sum_{0 < i < j}c^{i}\circ (\tau^i \cdot c^{j - i}) = c^j\circ (\tau^j \cdot d) +  d\circ c^j.
$$ 
Let $\zeta^n := -\sum_{0 < i < n}c^{i}\circ (\tau^i \cdot c^{n - i})$. 
Then
\begin{align*}
d \zeta^n =& - \sum_{0<i<n}  dc^i (\tau^i \cdot c^{n - i})
= \sum_{0<i<n} \left(c^i (\tau^i \cdot d)  (\tau^i \cdot c^{n - i}) - \zeta^i (\tau^i \cdot c^{n - i})\right) \\
=& \sum_{0<i<n} \left( c^i (\tau^i \cdot \zeta^{n - i}) - c^i (\tau^i \cdot c^{n - i})  (\tau^n \cdot d) - \zeta^i (\tau^i \cdot c^{n - i}) \right) \\
=& \sum_{0<i<n} \left( c^i (\tau^i \cdot \zeta^{n - i})- \zeta^i (\tau^i \cdot c^{n - i}) \right) - \sum_{0<i<n} c^i (\tau^i \cdot c^{n - i})  (\tau^n \cdot d)  \\
=& \sum_{0<i<n} \left( c^i (\tau^i \cdot \zeta^{n - i})- \zeta^i (\tau^i \cdot c^{n - i}) \right)
+ \zeta^n  (\tau^n \cdot d). 
\end{align*}
On the other hand, 
\begin{align*}
&\sum_{0<i<n} \left( c^i (\tau^i \cdot \zeta^{n - i})- \zeta^i (\tau^i \cdot c^{n - i}) \right) 
= \sum_{0<i<n}  c^i (\tau^i \cdot \zeta^{n - i})- \sum_{0<i<n} \zeta^i (\tau^i \cdot c^{n - i})  
\\
=& \sum_{0<i<n-1}  c^i (\tau^i \cdot \zeta^{n - i})- \sum_{1<i<n} \zeta^i (\tau^i \cdot c^{n - i})  
+ \left(c^{n-1}(\tau^{n-1} \cdot \zeta^1) - \zeta^1 (\tau \cdot c^{n-1})\right) 
\\
=& \sum_{0<i<n-1}  c^i (\tau^i \cdot \zeta^{n - i})- \sum_{1<i<n} \zeta^i (\tau^i \cdot c^{n - i}) 
&\text{(by (\ref{twist}))}
\\
=& -\sum_{0<i<n-1} \sum_{0<t<n - i} c^i(\tau^i \cdot c^{t}) (\tau^{t + i}\cdot c^{n - i - t}) 
+ \sum_{1<i<n}\sum_{0<l<i}c^{l} (\tau^{l} \cdot c^{i - l})  (\tau^{i} \cdot c^{n - i})   \\
=& 0. 
\end{align*}
Thus $\zeta^n : \tau^n P_\bullet(-nm) \rightarrow P_\bullet[-2n+2]$ is a chain map, and so $(\zeta^n_i)_{i\geq 0}$ is a chain map over $\overline{\zeta^n_0}=0: {\rm Coker}(\tau^n\cdot d_1)\rightarrow {\rm Coker}(d_{2n-1})$ which is induced by the commutative diagram
$$
\xymatrix{
\tau^n P_\bullet(-nm)_1 \ar[r]^-{\zeta^n_1} \ar[d]_-{\tau^n \cdot d_1} &   P_\bullet[-2n+2]_1 \ar[d]^-{d_{2n-1}} \\
\tau^n P_\bullet(-2m)_0 \ar[r]^-{\zeta^n_0} &  P_\bullet[-2n+2]_0. 
}
$$
Then there exists a sequence of morphisms 
$c^n = (c^n_l : \tau^n P_\bullet(-nm)_l \rightarrow P_\bullet[-2n+2]_{l+1})_{l\in \mathbb{Z}}$ such that $\zeta^n = c^n  (\tau^n \cdot d)  + d c^n$ by comparison theorem.
\end{proof}

\begin{lemma}\label{R1}
Keep the notation as in Lemma \ref{R0}. For $n\in\mathbb{N}$,
\[
\sum_{0 \leq i \leq n} c^i (\tau^i \cdot c^{n - i}) =
\begin{cases}
0,   & \text{if } n \geq 2, \\
\zeta^1, & \text{if } n = 1.
\end{cases}
\]
\end{lemma}

\begin{proof}
The case $n = 1$ is immediate.
If $n\geq 2$, then by Lemma \ref{R0}, we have
$$
-\sum_{0 < i < n}c^{i} (\tau^i \cdot c^{n - i}) = c^n (\tau^n \cdot d)  + d c^n,
$$ 
hence  $\sum_{0 \leq i \leq n} c^i (\tau^i \cdot c^{n - i}) = 0$. 
\end{proof}

For $l \geq 0$, let $e_l$ denote a generator of the free $B$-$B$-bimodule $Be_l$, and for $l \neq t$, $Be_l$ will be considered as distinct from $Be_t$. Set  $c^0 = d$.
Define a sequence 
\begin{equation} \label{miniB}
T_\bullet : \ \ \ \cdots \to T_i \xrightarrow{\partial_i} \cdots \to T_2 \xrightarrow{\partial_2} T_1 \xrightarrow{\partial_1} T_0 \to 0,
\end{equation}
where
$
T_i = \bigoplus\limits_{ 0\leq k \leq \left\lfloor \frac{i}{2}\right\rfloor }P_{i-2k}(-km)  \otimes_A B e_{2k},
$
and 
\[
\partial_i
=\begin{pmatrix}
c^0_{i}\otimes \theta_{00} &  c^1_{i-2}\otimes \theta_{10}  &  c^2_{i-4}\otimes \theta_{20}  &  c^3_{i-6}\otimes \theta_{30}  & \cdots & c^{t_i}_{i-2t_i}\otimes \theta_{t_{i}0}  \\
0 & (\tau \cdot c^0_{i-2})\otimes \theta_{11} & (\tau \cdot c^1_{i-4}) \otimes \theta_{21} & (\tau \cdot c^2_{i-6}) \otimes \theta_{31} & \cdots & (\tau \cdot c^{t_i-1}_{i-2t_i})\otimes \theta_{t_{i}1}  \\
0 & 0 & (\tau^2 \cdot c^0_{i-4}) \otimes \theta_{22} & (\tau^2 \cdot c^1_{i-6}) \otimes \theta_{32} & \cdots & (\tau^2 \cdot c^{t_i-2}_{i-2t_i}) \otimes \theta_{t_{i}2} \\
0 & 0 & 0 & (\tau^3 \cdot c^0_{i-6}) \otimes \theta_{33} & \cdots & (\tau^3 \cdot c^{t_i-3}_{i-2t_i}) \otimes \theta_{t_{i}3} \\
\vdots & \vdots & \vdots & \vdots & \ddots & \vdots
\end{pmatrix}
 \]
with $t_i= \left\lfloor \frac{i}{2}\right\rfloor$ and $\theta_{l_1l_2}: Be_{2l_1} \rightarrow Be_{2l_2}$, $e_{2l_1}\mapsto e_{2l_2}$.

Let
$$
\vartheta = (\vartheta_k)_{k \geq 0} = (v_k \otimes e_{2k})_{k\geq 0} \in T_i = \bigoplus_{\left\lfloor \frac{i}{2}\right\rfloor \geq k \geq 0}P_{i-2k}(-km)  \otimes_A B e_{2k}.
$$
Then $\vartheta_k = 0$ in $P_{i-2k}(-km)  \otimes_A B e_{2k}$ if and only if there is $v' \in P_{i-2k}(-km)$ such that $v_k = v'f$. We define the {\it support} of $\vartheta$ as 
$\supp (\vartheta): = \{j\in \mbN \mid \vartheta_j\neq 0 \}$. Then $\vartheta = 0$ precisely when its support is empty. 

\begin{theorem}\label{T1}
Let $A$ be a connected graded algebra, $f \in A_m$ a regular normal element with $m \geq 1$, and $B = A/(f)$. 
Let $P_\bullet = (P_i, d_i) \xrightarrow{\sim} \Bbbk_A$ be a minimal resolution. 
Then the sequence $T_\bullet = (T_i, \partial_i)$  defined in (\ref{miniB}) is a free resolution of $\Bbbk_B$. 
If in addition $A$ is Koszul and $m\geq 2$, then $T_\bullet\xrightarrow{\sim} \Bbbk_B$ is a minimal resolution.
\end{theorem}

\begin{proof}
We divide the proof into $3$ steps. 

(i) By  Lemma \ref{R1}, we can check that $\partial_i\circ \partial_{i+1} = 0$
for $i\geq 1$. Thus $T_\bullet$ is a complex. 

(ii) We now show
$\ker \partial_i = \im \partial_{i+1}$ for $i\geq 1$. 
Let 
$
\vartheta = (v_k \otimes e_{2k})_{k\geq 0} \in T_i,
$
such that $\partial_i(\vartheta) = 0$. Let $l = \max\,\supp(\vartheta)$.
Then there are two cases we need to discuss: 

Case 1. Either $i$ is odd, or $i$ is even and $l \leq  \frac{i}{2} - 1$. 
Then $\partial_i(\vartheta) = 0$ implies that 
$$
\left((\tau^{l} \cdot d_{i-2l})(v_{l})\right) \otimes e_{2l} = 0
$$
in $P_{i- 2l - 1}(-l m ) \otimes_A B e_{2l}$. It follows that there is $v' \in P_{i-2l-1}(- l m)$
such that $(\tau^{l} \cdot d_{i-2l}) (v_{l}) = v'f$. Then
\begin{align*}
(\tau^{l} \cdot d_{i - 2l - 1}) (v') f =& (\tau^{l} \cdot d_{i - 2l - 1}) (v' f ) \\
=& (\tau^{l} \cdot d_{i - 2l - 1}) \circ (\tau^{l} \cdot d_{i - 2l}) (v_{l}) \\
=& (\tau^{l} \cdot (d_{i - 2l - 1} d_{i - 2l})) (v_{l}) = 0
\end{align*}
in $P_{i-2l-2}(-l m )$. Thus $(\tau^{l} \cdot d_{i - 2l - 1}) (v') = 0$. Since $P_\bullet(-lm)$ is a minimal resolution, 
there exists $v'' \in P_{i-2l}(-l m )$ such that $(\tau^{l} \cdot d_{i - 2l}) (v'') = v'$. 
It follows that 
$
(\tau^{l} \cdot d_{i - 2l} )(v_{l} - v''f) = 0,$
and so there exists $x \in P_{i - 2l + 1 }(-lm)$ such that $(\tau^{l} \cdot d_{i - 2l + 1}) (x) =v_{l}  - v''f$.
Let $\omega^{} = (w_k)_{k\geq 0} \in T_{i+1}$ where 
$$
w_k = \begin{cases} x \otimes e_{2l} & \text{if } k = l,  \\ 0 & \text{otherwise}. \end{cases}
$$
Then $\max\,\supp(\vartheta - \partial_{i+1}(\omega^l)) < l_0:=l$. Similarly, we can construct $\omega^{l_1} \in T_{i+1}$ such that the maximal value of $\vartheta - \partial_{i+1}(\omega^l_0) - \partial_{i+1}(\omega^{l_1})$ is smaller. Inductively, we have
$\omega^{l_0}, \dots, \omega^{l_s} \in T_{i+1}$ such that $\vartheta - \sum_{j=1}^{s} \partial_{i+1}(\omega^{l_j}) = \vartheta - \partial_{i+1}\left(\sum_{j=1}^{s}  \omega^{l_j} \right) = 0$. 

Case 2. $i$ is even and $l=\frac{i}{2}$. Then $\partial_i(\vartheta) = 0$ implies that 
\begin{align*}
&((\tau^{l-1} \cdot d_{i - 2(l-1)}) (v_{l-1})) \otimes e_{2(l-1)} + ((\tau^{l-1} \cdot c^1_{i - 2l}) (v_{l})) \otimes e_{2(l-1)} \\
=&((\tau^{l-1} \cdot d_{2}) (v_{l-1}) + (\tau^{l-1} \cdot c^1_{0}) (v_{l})) \otimes e_{i-2}
= 0.
\end{align*}
in $P_{1}(- (l -1)m ) \otimes_A B e_{i-2}$. Here, $v_{l-1}$ may be $0$.
Then 
$
(\tau^{l-1} \cdot d_{2}) (v_{l-1}) + (\tau^{l-1} \cdot c^1_{0}) (v_{l}) = v'f
$
for some $v' \in P_{1}(- (l -1)m )$. This further leads to 
\begin{align*}
(\tau^{l-1} \cdot d_1) (v')f = (\tau^{l-1} \cdot d_{1}) \circ ( \tau^{l-1} \cdot c^1_{0} )(v_{l}) 
= \tau^{l-1} \cdot (d_1 c^1_{0} )(v_{l}) 
= (\tau^{l-1} \cdot \zeta^1_{0})(v_{l}) = fv_{l}.
\end{align*}
It follows that $v_{l} = (\tau^{l} \cdot d_1)(v')$. 
Let $\mu = (u_k)_{k\geq 0} \in T_{i+1}$ where
$$
u_k = \begin{cases} v' \otimes e_{l} & \text{if } k = l, \\ 0 & \text{otherwise}. \end{cases}
$$
Then  $\max\,\supp(\vartheta - \partial_{i+1} (\mu)) < l$. Thus we can use the method in Case 1 to construct $\omega \in T_{i+1}$ such that 
$
\vartheta - \partial_{i+1}  (\mu) - \partial_{i+1}  (\omega) = 
\vartheta - \partial_{i+1}  (\mu + \omega)
= 0. 
$

(iii) By applying $ - \otimes_A Be_0$ to the exact sequence  $P_1 \xrightarrow{d_1}  P_0 \xrightarrow{\varepsilon}  \Bbbk_A \to 0$, we get an exact sequence
$$
T_1 \xrightarrow{\partial_1} T_0 \xrightarrow{\widetilde{\varepsilon}} k_B \to 0
$$
where $\widetilde{\varepsilon} = \varepsilon \otimes_A Be_0$. Thus $T_\bullet$ is a free resolution of $\Bbbk_B$ with augmentation 
$\widetilde{\varepsilon}: T_\bullet \to \Bbbk_B$. 

If in addition $A$ is Koszul and $m\geq 2$, then $-im-j< (-2i+2)-(j+1)=-2i-j+1$.  This implies that $c^i_j: \tau^i P_\bullet(-im)_j \rightarrow P_\bullet[-2i+2]_{j+1}$ is a matrix over $A_{\geq 1}$ for all $i,j$. Thus, $T_\bullet$ is a minimal resolution. 
\end{proof}

If $A$ is not Koszul or $m = 1$, the free resolution $T_\bullet$ of $\Bbbk_B$ provided by Theorem \ref{T1} may fail to be minimal.

\begin{example}
Let $A = \Bbbk[x]$ with $\deg x = m \geq 1$, and $B = A/(x) = \Bbbk$. We have a free resolution
$$
\cdots \rightarrow \Bbbk(-2m)\xrightarrow{1} \Bbbk(-2m) \xrightarrow{0} \Bbbk(-m)\xrightarrow{1}\Bbbk(-m) \xrightarrow{0} \Bbbk \to \Bbbk_B \rightarrow 0
$$
provided by Theorem \ref{T1}.
This resolution is not minimal for every $m \in \mathbb{N}$.
\end{example}

\section{Main results} \label{sec-mainres}

In this section, $A = \Bbbk\langle x_1,\dots,x_n \rangle/(f_1,\dots, f_r) $ is a Koszul algebra, and $B = A/(f)$, where $f\in A_2$ a regular normal element. Then $B$ is also a Koszul algebra.

By Theorem \ref{T1}, starting from a linear resolution $P_\bullet = (P_i, d_i)$ of $\Bbbk_A$,
we obtain a linear resolution of $\Bbbk_B$:
\begin{align*}
T_\bullet :  \ \ \  \cdots \to T_i \xlongrightarrow{\partial_i}  \cdots \to T_2 \xlongrightarrow{\partial_2}  T_1 \xlongrightarrow{\partial_1} T_0 = B \to  0. 
\end{align*}
For convenience, we write
$$
\partial_i 
=\begin{pmatrix}
d_i& c^1_{i-2}  &  c^2_{i-4}  & \cdots & \ \ \ c^{t_i}_{i-2t_i}  \\
0 & \tau \cdot d_{i-2} & \tau \cdot c^1_{i-4}  & \cdots & \tau \cdot c^{t_i-1}_{i-2t_i}  \\
0 & 0 & \tau^2 \cdot d_{i-4}  & \cdots & \tau^2 \cdot c^{t_i-2}_{i-2t_i} \\
\vdots & \vdots & \vdots & & \vdots
\end{pmatrix}
$$
where $t_i= \left\lfloor \frac{i}{2}\right\rfloor$, $\tau$ is the inverse of normalizing automorphism of $f$, and the maps $c^s_l$ are as defined in Section \ref{sec-resquo}.
Similarly, from a linear resolution $Q_\bullet = (Q_i, h_i)\xrightarrow{\sim} {_A}\Bbbk$, we obtain a linear resolution $K_\bullet = (K_i, \delta_i)\xrightarrow{\sim}{_B}\Bbbk$. 

\begin{remark}
For any $s\in\mathbb{Z}$, $\tau^sP_\bullet$ is also a linear resolution of $\Bbbk_A$ by Lemma \ref{lem-tplin}. Thus there exists a chain isomorphism 
$(\phi_i)_{i\in\mathbb{Z}} :  \tau^s P_\bullet \to P_\bullet$. 
Therefore,
$
\rank_\Bbbk(\tau^s \cdot d_i)_p=\rank_\Bbbk(d_i)_p, \ \forall p\in \mathbb{P}^{n-1}.
$
\end{remark}

The following lemma is direct.

\begin{lemma} \label{lem-rkbti}
The following hold.
\begin{itemize}
\item[(1)] $\rank_{B} T_i = \sum_{j = 0}^{\left\lfloor \frac{i}{2}\right\rfloor} \rank_A P_{i - 2j}$, for $i \geq 0$.
\item[(2)] $\rank_\Bbbk (\partial_i)_p \geq \sum_{j = 0}^{\left\lfloor \frac{i-1}{2}\right\rfloor} \rank_\Bbbk (d_{i - 2j})_p$, $\forall p \in \mathbb{P}^{n-1}$, for $i \geq 0$.
\end{itemize}
If in addition, $A$ is right point-exact up to degree $\ell \geq 1$. Then
$$
\rank_\Bbbk (\partial_i)_p \geq  \sum_{j=1}^i (-1)^{j+1}\rank_B T_{i-j},  \ \forall p \in X_A
$$
where $0 \leq i \leq \ell +1$.
\end{lemma}

\begin{proposition} \label{prop-leq2}
If $A$ is semi-standard and point-exact up to degree $2$. Then $B$ satisfies the (G1) condition. 
\end{proposition}

\begin{proof}
By Lemma \ref{lem-E1}, $A$ satisfies the (G1) condition, and $r + 1 \geq n$. Let $(E,\sigma)$ be the geometric pair of $A$. 
It suffices to show that $B$ is semi-standard and point-exact at degree $1$.

Let $q \in X_{\partial_2}$, then there exists $p \in \mathbb{P}^{n-1}$ such that 
$$
(\partial_1)_p (\partial_2)_q = (f_1(p,q), \dots, f_r(p,q), f(p,q)) = 0. 
$$
Since $X_{\partial_2} \subset X_{d_2} = E$, we have $p \in X_{h_2} = E$ such that $\sigma(p) = q$.
Then $(\delta_3)_p (\delta_2)_q = 0$ since $\delta_3 \delta_2 = 0$ as a matrix over $B$. 
On the other hand, by (left version of) Lemma \ref{lem-rkbti}, 
\begin{equation*}
\begin{aligned}
\rank_\Bbbk (\delta_2)_q &\geq \rank_B K_1 - \rank_B K_0 = n - 1, \\ 
\rank_\Bbbk (\delta_3)_p &\geq  \rank_B K_2 - \rank_B K_1 + \rank_B K_0. 
\end{aligned}
\end{equation*}
Then by Lemma \ref{lem-rk}, we have
\begin{equation*}
\begin{aligned}
\rank_\Bbbk(\delta_2)_q \leq \rank_B K_2 - (\rank_B K_2 - \rank_B K_1 + \rank_B K_0) =  n-1. 
\end{aligned}
\end{equation*}
It follows that $\rank_\Bbbk (\delta_2)_q = n-1$. Since $\delta_2 \in \mathbb{M}_{(r + 1) \times n}(B_1)$ and $r + 1 \geq n$, $q \in X_{\delta_2}$. 
Therefore $X_{\partial_2} \subset X_{\delta_2}$. Similarly, we can prove $X_{\delta_2} \subset X_{\partial_2}$, and $\rank_\Bbbk(\partial_2)_p = n- 1$ for all $p \in X_{\delta_2}$. The result follows.
\end{proof}

\begin{remark}
Let $A$ and $B$ as in Proposition \ref{prop-leq2},  
and let $(E, \sigma)$ and $(E_B, \sigma_B)$ be the geometric pairs of $A$ and $B$ respectively. 
Then $E_B \subset E$, and $\sigma_B = \sigma|_{E_B}$. 
\end{remark}

\begin{theorem} \label{thm-main}
If $A$ is semi-standard and point-exact up to degree $\ell \geq 2$. Then $B$ is also semi-standard and point-exact up to  degree $\ell$. 
\end{theorem}

\begin{proof}
By Lemma \ref{lem-E1} and Proposition \ref{prop-leq2}, both $A$ and $B$ satisfy the (G1) condition, thus $B$ is semi-standard and point-exact at degree 1. Let $(E, \sigma)$ and $(E_B, \sigma_B)$ be the geometric pairs of $A$ and $B$ respectively. We give the proof that $B$ is right point-exact up to degree $\ell$. The proof of left version is similar. 

Let $p \in E_B$. 
By Lemma \ref{lem-rkbti}, we have
$$
\rank_\Bbbk(\partial_i)_p \geq  \sum_{j=1}^i (-1)^{j+1}\rank_B T_{i-j}, \ \text{where} \ 0 \leq i \leq \ell+1.
$$
Since $\partial_{i-1} \partial_i = 0$ as a matrix over $B$, we have $(\partial_{i-1})_{\sigma_B^{-1}(p)} (\partial_i)_{p} = 0$. Then by Lemma \ref{lem-rk},  
$$
\rank_\Bbbk(\partial_i)_p \leq \rank_B T_{i-1} - \left( \sum_{j=1}^{i-1} (-1)^{j+1}\rank_B T_{(i - 1)-j} \right) = \sum_{j=1}^i (-1)^{j+1}\rank_B T_{i-j}. 
$$
Thus $\rank_\Bbbk(\partial_i)_p =  \sum_{j=1}^i (-1)^{j+1}\rank_B T_{i-j}$. By Lemma \ref{lem-rkdi}, the desired conclusion follows.
\end{proof}

If $f\in A_2$ is not a regular normal element, then $B=A/(f)$ is not necessarily semi-standard.

\begin{example}
	Let $A = k\langle x,y,z \rangle /(f_1,f_2,f_3)$ where 
	$$
	f_1 = xy+yx+2z^2, \, f_2 = yz + zy + 2x^2, \, f_3 = zx + xz + 2y^2.
	$$ 
	Then $A$ is a $3$-dimensional quantum polynomial algebra.
	Let $f = xy \in A_2$ which is not a normal element, and $B = A/(f)$. 
 Then $(f_1,f_2,f_3,f) = (x,y,z)M = N(x,y,z)^t$, where 
	$$
	M = \begin{pmatrix}
		y & 2x & z & y\\
		x  &  z &  2y & 0 \\
		2z  &  y  &  x  &  0
	\end{pmatrix},\ \ 
	N = \begin{pmatrix}
		y  &  x  & 2z\\
		2x  &  z  &  y \\
		z  &  2y  &  x \\
		0  &  x  &  0
	\end{pmatrix}.
	$$
This implies that 
	$$
X_B = X^3_M = \mcZ(10xyz-2(y^3+x^3+z^3), y(xy-2z^2), y(zx-2y^2), y(x^2-4yz)) \subset \mbP^2
	$$ 
	and 
	$$
{_B}X = X^3_N = \mcZ(10xyz-2(y^3+x^3+z^3), x(xy-2z^2), x(yz-2x^2), x(4xz-y^2))\subset \mbP^2.
	$$ 
We can see that $(1\mathpunct{:}0\mathpunct{:}-1)\in X_B$, but  $(1\mathpunct{:}0\mathpunct{:}-1)\notin {_B}X$. It follows that $ X_B \neq {_B}X$, and so $B$ is not semi-standard.
\end{example}

\section{Skew polynomial algebras}\label{sec6}

Skew polynomial algebras, as defined below, form an important class of quantum polynomial algebras \cite{BDL, R,V10}. 

\begin{definition} \label{defn-spa}
A quadratic algebra $A$ is called a {\it skew polynomial algebra of dimension $n$} if 
$$
A = \Bbbk \langle x_1, \dots, x_n \rangle /(x_j x_i - q_{ij} x_i x_j)_{1 \leq i,j \leq n} 
$$
with $\deg x_i = 1$, and all entries of $\mathcal{Q} =(q_{ij}) \in \mathbb{M}_{n}(\Bbbk)$ are nonzero such that $q_{ii} = 1$ and $q_{ji} = q_{ij}^{-1}$. 
\end{definition}

Let $A$ be an $n$-dimensional skew polynomial algebra defined by a matrix $\mathcal{Q} =(q_{ij}) \in \mathbb{M}_{n}(\Bbbk)$. Let
$$
\mathbf{q} =  \left(q_{1(n+1)}, \dots, q_{n(n+1)}\right)^t \in \mathbb{M}_{n \times 1}(\Bbbk)
$$
where $q_{i(n+1)} \neq 0$ for $i = 1, \dots,n$. Let
$$
\mathcal{Q}^{\wedge, \mathbf{q}} : = \begin{pmatrix} \mathcal{Q} & \mathbf{q} \\ \tilde{\mathbf{q}} & 1\end{pmatrix} \in \mathbb{M}_{(n+1) \times (n+1)}(\Bbbk)
$$
where 
$$
\tilde{\mathbf{q}} : = \left(q_{1(n+1)}^{-1}, \dots,  q_{n(n+1)}^{-1}\right) \in \mathbb{M}_{1 \times n}(\Bbbk). 
$$
Then there is an $(n+1)$-dimensional skew polynomial algebra $A^{\wedge, \mathbf{q}}$ defined by the matrix $Q^{\wedge, \mathbf{q}}$,
which can be realized as an Ore extension $A^{\wedge, \mathbf{q}} = A[x_{n+1}; \mu^{\mathbf{q}}]$,
where
$$
\mu^{\mathbf{q}} := \diag \{ q_{1(n+1)}, \dots, q_{n(n+1)} \} \in \Aut A. 
$$

Let $P_\bullet = (P_i,d_i) \xrightarrow{\sim} \Bbbk_A$ be a linear resolution. Since $A$ is an $n$-dimensional quantum polynomial algebra, we have $P_i = A(-i)^{\oplus C^n_i}$, where  $C^n_i = \binom{n}{i}$ is the binomial coefficient.  By applying $- \otimes_A A^{\wedge, {\mathbf{q}}}$ to $P_\bullet$, we obtain a free resolution $Z_\bullet  \xrightarrow{\sim} \Bbbk[x_{n+1}]_{A^{\wedge, \mathbf{q}}}$. 
Let 
$
\varphi: Z_\bullet(-1) \to Z_\bullet
$
be the chain map over the left multiplication 
$$
x_{n+1} \cdot: \Bbbk[x_{n+1}]_{A^{\wedge, \mathbf{q}}}(-1) \to \Bbbk[x_{n+1}]_{A^{\wedge, \mathbf{q}}}. 
$$
Then $\varphi_0 = x_{n+1}$ and $\varphi_i = 0$ for $i > n$. We call $\varphi$ a {\it $\mathbf{q}$-extended map of $P_\bullet$}.

Define $P^{\wedge, \mathbf{q}}_{\bullet} = (P^{\wedge, \mathbf{q}}_i, d^{\wedge, \mathbf{q}}_i)$ to be the mapping cone of 
$\varphi$, i.e.,
\begin{align}\label{list:resolution}
P^{\wedge, \mathbf{q}}_i 
= (A^{\wedge, \mathbf{q}})(-i)^{\oplus C^{n}_{i-1}} \oplus (A^{\wedge, \mathbf{q}})(-i)^{\oplus C^{n}_{i}} = (A^{\wedge, \mathbf{q}})(-i)^{\oplus C^{n+1}_i}, \  
d^{\wedge, \mathbf{q}}_{i} = \begin{pmatrix}
-d_{i-1} &  0   \\
\varphi_{i-1} &  d_i 
\end{pmatrix} 
\end{align}
for $i \in \mathbb{Z}$. 
Then $P^{\wedge, \mathbf{q}}_{\bullet} \xrightarrow{\sim} \Bbbk_{A^{\wedge, \mathbf{q}}}$ is a minimal resolution (\cite[Lemma 2.4]{SG21}). We call the complex $P^{\wedge, \mathbf{q}}_{\bullet}$ a {\it cone extension} of $P_\bullet$ by $\mathbf{q}$. 

Consider the matrix $\mathcal{Q} = (q_{ij})$ associated to $A$. Let 
\begin{align*}
\mathbf{q}_i = \left(q_{1(i+1)}, \dots q_{i(i+1)}\right)^t \in \mathbb{M}_{(i+1)\times 1}(\Bbbk), \ 1\leq i \leq n-1,
\end{align*}
and 
$$
C_\bullet: \ \ \ 0\rightarrow \Bbbk[x_1](-1) \xrightarrow{x_1\cdot} \Bbbk[x_1]\rightarrow 0
$$
the minimal resolution of $\Bbbk_{R}$ where $R = \Bbbk[x_1]$. Then the iterated cone extension
$$
\widetilde{P}_\bullet : = (((C^{\wedge, \mathbf{q}_1}_\bullet)^{\wedge, \mathbf{q}_2}) \cdots)^{\wedge, \mathbf{q}_{n-1}}
$$
is a linear resolution of $\Bbbk_A$. We call $\widetilde{P}_\bullet \xrightarrow{\sim} \Bbbk_A$ the {\it canonical resolution}. 

\begin{lemma} 	\label{lem-spa}
Keep the notation as above.
If $P _\bullet = \widetilde{P}_\bullet \xrightarrow{\sim} \Bbbk_A$ is the canonical resolution, then the $\mathbf{q}$-extended map $\varphi: Z_\bullet(-1) \to Z_\bullet$ of $P_\bullet$ has a diagonal presentation associated to $x_{n+1}$, which is to say, 
$$
\varphi_i = x_{n+1} \diag\{ a_{i1}, a_{i2},\dots,a_{iC^n_i} \}, \ 0 \leq i \leq n, 
$$
where $0 \neq a_{ij} \in \Bbbk$ for all $i,j$.  
\end{lemma}

\begin{proof}
We proceed by induction on the dimension of the skew polynomial algebra. 

If $n=1$, then 
\begin{equation}\label{skew1}
P_\bullet: \  \ \ 0\rightarrow \Bbbk[x_1](-1) \xrightarrow{x_1\cdot} \Bbbk[x_1]\rightarrow 0.
\end{equation} and 
$\mathbf{q} = (q_{12})$ where $0 \neq q_{12}\in \Bbbk$. Then $A^{\wedge,\mathbf{q}} = \Bbbk\langle x_1, x_2\rangle/(x_2 x_1 - q_{12}x_1 x_2)$ and 
$$
Z_\bullet: \ \ \ 0\rightarrow A^{\wedge,\mathbf{q}} (-1) \xrightarrow{x_1\cdot} A^{\wedge,\mathbf{q}}\rightarrow 0.
$$
The map $\varphi: Z_\bullet(-1) \to Z_\bullet$ is given by $\varphi_0 = x_2$ and $\varphi_1 = q_{12} x_2$, which clearly has the required diagonal from.

Assume the lemma holds for all skew polynomial algebras of dimension less than $n$.
Consider the $(n-1)$-dimensional skew polynomial algebra $A^\vee$ defined by the leading principal submatrix $\mathcal{Q}^\vee$  of $\mathcal{Q}$ of order $n-1$, i.e.,
$$
A^\vee = \Bbbk\langle x_1, \dots, x_{n-1} \rangle / (x_jx_i - q_{ij}x_ix_j)_{1 \leq i,j \leq n-1}.
$$
Let $P^\vee_\bullet \xrightarrow{\sim} \Bbbk_{A^\vee}$ be the canonical resolution,
and let 
$$
\psi: (P^\vee_\bullet \otimes_{A^\vee} A) (-1) \to P^\vee_\bullet \otimes_{A^\vee} A
$$ 
be the $\mathbf{q}_{n-1}$-extended map of $P^\vee_\bullet$ where $\mathbf{q}_{n-1} = \left(q_{1n}, \dots q_{(n-1)n}\right)^t$. 
Then $A = A^\vee[x_n; \mu^{\mathbf{q}_{n-1}}]$, and $P_\bullet = (P^\vee_\bullet)^{\wedge, \mathbf{q}_{n-1}}$.
By the induction hypothesis, $\psi$ has a diagonal presentation associated to $x_n$. 

Now set 
$$
\mathbf{k} = \left( q_{1(n+1)}, \dots, q_{(n-1)(n+1)} \right)^t \in \mathbb{M}_{(n-1)\times 1}(\Bbbk),
$$
and define $B := A^\vee[x_{n+1}; \mu^{\mathbf{k}}]$ (look out the variable).  
Then the $\mathbf{k}$-extended map
$$
\theta: (P^\vee_\bullet \otimes_{A^\vee} B) (-1) \to P^\vee_\bullet \otimes_{A^\vee}B
$$
of $P^\vee_\bullet$ is a chain map over the left multiplication $x_{n+1}\cdot: \Bbbk[x_{n+1}]_B \to \Bbbk[x_{n+1}]_B$. 
By the induction hypothesis, $\theta$ also has a diagonal presentation associated to $x_{n+1}$. 

Define a sequence of morphisms
$\varphi = (\varphi_i)_{i \in \mathbb{Z}}: Z_\bullet(-1) \to Z_\bullet$ by
\begin{align*}
\varphi_i := &
\begin{pmatrix}
q_{n(n+1)}\theta_{i-1} \otimes_B A^{\wedge, \mathbf{q}} & 0\\
0  &  \theta_i  \otimes_B A^{\wedge, \mathbf{q}}
\end{pmatrix} 
\end{align*}
for $0 \leq i \leq n$. 
Since $\theta$ is a chain map
and
$$
(\theta \otimes_{B} A^{\wedge, \mathbf{q}}) \circ (\psi \otimes_{A} A^{\wedge, \mathbf{q}}) = q_{n(n+1)} (\psi \otimes_{A} A^{\wedge, \mathbf{q}}) \circ (\theta \otimes_{B} A^{\wedge, \mathbf{q}}),
$$
the map $\varphi$ is a chain map. Moreover, $\varphi_0 = x_{n+1}$ implies that $\varphi$ is a chain map over the left multiplication $x_{n+1} \cdot: \Bbbk[x_{n+1}]_{A^{\vee,\mathbf{q}}}(-1) \to \Bbbk[x_{n+1}]_{A^{\vee,\mathbf{q}}}$. 
Thus $\varphi: Z_\bullet(-1) \to Z_\bullet$ is a $\mathbf{q}$-extended map of $P_\bullet$, which has a diagonal presentation associated to $x_{n+1}$. 
\end{proof}

\begin{remark}\label{rem-spa}
Let $A$ be an $n$-dimensional skew polynomial algebra, and
$P_\bullet = (P_i, d_i) \xrightarrow{\sim} \Bbbk_A$ the canonical resolution. By (\ref{list:resolution}) and Lemma \ref{lem-spa}, for every $i$, the entries of $d_i$ are monomials, and for each variable $x_j$, there is $0 \neq c \in \Bbbk$ such that $c x_j$ appears as an entry of $d_i$. 
\end{remark}

\begin{theorem} \label{thm-spa}
Skew polynomial algebras of dimension $\geq 2$ satisfy the point-exact condition.
\end{theorem}

\begin{proof}
We proceed by induction on the dimension of the skew polynomial algebra. For $n = 2,3,4$, the result follows form Proposition \ref{prop.34qec} (see Example \ref{exm-2dimqpa}). 

Assume the result holds for all skew polynomial algebras of dimension less that $n$. 
Let $A$, $A^\vee$, 
$P_{\bullet}  = (P_i, d_i)\xrightarrow{\sim} \Bbbk_{A}$,  
$P^\vee_{\bullet} = (P^\vee_i, d^\vee_i) \xrightarrow{\sim} \Bbbk_{A^\vee}$ 
be as in the proof of Lemma \ref{lem-spa}. 
Then $A = A^\vee[x_n; \mu^{\mathbf{q}_{n-1}}]$ and $P_\bullet = (P^\vee_\bullet)^{\wedge, \mathbf{q}_{n-1}}$ where $\mathbf{q}_{n-1} = \left(q_{1n}, \dots q_{(n-1)n}\right)^t$. 
In particular, we have
$$
d_i = 
\begin{pmatrix}
-d^\vee_{i-1} &  0   \\
\psi_{i-1} &  d^\vee_i 
\end{pmatrix}
$$
for $i \in \mathbb{Z}$, where $\psi = (\psi)_{i \in \mathbb{Z}}$ is the $\mathbf{q}_{n-1}$-extended map of $P^\vee$ which has a diagonal presentation associated to $x_n$.

By Remark \ref{rem-asg}, both $A$ and $A^\vee$ satisfy the (G1) condition. 
Let $(E_A, \sigma_A)$ and $(E_{A^\vee}, \sigma_{A^\vee})$ be the geometric pairs of $A$ and $A^\vee$ respectively. 
Let 
$$
V  = \langle x_1,\dots,x_n \rangle, \ V' = \langle x_1,\dots,x_{n-1} \rangle.
$$
Take a point $p=(p_1\mathpunct{:} \, \cdots \, \mathpunct{:} p_n)\in E_{A} \subset  \mathbb{P}(V^*)$, and let $q = \sigma_A(p) = (q_1 \mathpunct{:}  \, \cdots \,  \mathpunct{:} q_n)$. 

If $p=(0\mathpunct{:} \, \cdots \, \mathpunct{:}0\mathpunct{:}1)$, 
then since $\psi$ is a chain map, we have
$$
0 = (d^\vee_1)_p(\psi_1)_q = (\psi_0)_p(d^\vee_1)_q = (x_n)_p(d^\vee_1)_q = (d^\vee_1)_q.
$$ 
By Remark \ref{rem-spa}, $(d^\vee_1)_q = 0$ implies that $q = p$. 
Since $\psi$ has a diagonal presentation associated to $x_{n}$, we have
$$
\rank_\Bbbk (d_{i})_{p} = \rank_\Bbbk \begin{pmatrix}
    	0 & 0    \\
    	(\psi_{i-1})_{p} &  0
    \end{pmatrix} = \rank_\Bbbk (\psi_{i-1})_{p}  = C^{n-1}_{i-1}
$$ 
for all $i\geq 1$.  

Now suppose $p_i \neq 0$ for some $1\leq i \leq n-1$, then $q_i\neq 0$ for some $1\leq i\leq n-1$ by the above discussion. Let 
$$
p' = (p_1\mathpunct{:}\, \cdots \,\mathpunct{:}p_{n-1}), q' = (q_1\mathpunct{:}\, \cdots \,\mathpunct{:}q_{n-1}) \in \mathbb{P}(V'^{*})
$$
be the truncations of $p,q$ respectively. 
We have $(d^\vee_1)_{p'}(d^\vee_2)_{q'}=0$ since $(d_1)_p(d_2)_q=0$. It follows that $p'\in E_{A^\vee}$ and $q'=\sigma_{A^\vee}(p')$. By assumption, we have
\begin{align*}
\rank_\Bbbk(d_i)_p + \rank_\Bbbk(d_{i+1})_q &
\geq \rank_\Bbbk(d^\vee_i)_{p'} + \rank_\Bbbk(d^\vee_{i-1})_{p'}+\rank_\Bbbk(d^\vee_{i+1})_{q'} + \rank_\Bbbk(d^\vee_{i})_{q'}\\
& = C^{n-1}_{i-1} + C^{n-1}_i \\
& = C^n_{i}
\end{align*} 
for all $i\geq 1$. One the other hand, $\rank_\Bbbk(d_i)_p + \rank_\Bbbk(d_{i+1})_{q} \leq C^n_i $ since $(d_i)_p (d_{i+1})_q = 0$. So $\rank_\Bbbk(d_i)_p + \rank_\Bbbk(d_{i+1})_q = C^n_i$ for all $i\geq 1$. 
Thus $A$ satisfies the right point-exact condition. The left point-exact condition is verified similarly. 
\end{proof}

\section{Appendix: Smooth $\pm 1$-skew quadric surfaces}

Let $A$ be a noetherian connected graded algebra, and $\tors A$ the full subcategory of $\grmod A$ consisting of finite dimensional modules. Define a quotient category $\qgr A := \grmod A / \tors A$. Following \cite{SVdB}, we call $\qgr A$ {\it smooth} if $\gld (\qgr A) < \infty$. 

We call a graded algebra $B = A/(f)$ a {\it $\pm 1$-skew quadric surface} if 
$$
A = \Bbbk\langle x_1, x_2, x_3, x_4 \rangle/(f_1, f_2, f_3, f_4, f_5, f_6)
$$ 
is a $4$-dimensional skew polynomial algebra defined by a matrix $\mathcal{Q} = (q_{ij}) \in \mathbb{M}_4(\Bbbk)$ with $q_{ij} = \pm 1$ for all $i,j$, and $f \in A_2$ is a regular normal element. By Proposition \ref{rem-asg}, Theorem \ref{thm-main} and Theorem \ref{thm-spa}, $B$ satisfies the (G1) condition and point-exact condition. For the rest, we consider $\pm 1$-skew quadric surfaces of the form $B = A/(x_1^2 + x_2^2 + x_3^2 + x_4^2)$. 

\begin{proposition} [\cite{Ue23}] \label{thm-ske4}
For a $\pm 1$-skew quadric surface of the form $B = A/(x_1^2 + x_2^2 + x_3^2 +x_4^2)$. The following statements hold. 
\begin{enumerate}
\item $\qgr B$ is smooth. 
\item The derived category $ \mathrm{D}^{\mathrm{b}}(\qgr B) $ is equivalent to exactly one of the following categories:
\begin{itemize}
\item[Case 1.] $ \mathrm{D}^{\mathrm{b}}(\qgr \Bbbk[x_1, x_2, x_3, x_4]/(x_1^2 + x_2^2 + x_3^2 + x_4^2))$. 
\item[Case 2.] $\mathrm{D}^{\mathrm{b}}(\qgr A/(x_1^2 + x_2^2 + x_3^2 + x_4^2))$ where $A$ is defined by
\begin{equation} \label{equ-case2}
\mathcal{Q} = 
\begin{pmatrix} 
1 & -1 & -1 & 1 \\ 
-1 & 1 & -1 & -1 \\
-1 & -1 & 1 & -1 \\
1 & -1 & -1 & 1 
\end{pmatrix}. 
\end{equation}
\item[Case 3.] $\mathrm{D}^{\mathrm{b}}(\qgr A/(x_1^2 + x_2^2 + x_3^2 + x_4^2))$ where $A$ is defined by
$
\mathcal{Q} = (q_{ij}) \in \mathbb{M}_4(\Bbbk)
$
with $q_{ij} = -1$ for all $i \neq j$. 
\end{itemize}
\item For any $B$ and $B'$, the following are equivalent:
\begin{enumerate}
\item $ \GrM B \cong \GrM B' $.
\item $ \qgr B \cong \qgr B' $.
\item $ \mathrm{D}^{\mathrm{b}}(\qgr B) \cong \mathrm{D}^{\mathrm{b}}(\qgr B') $.
\end{enumerate}
\end{enumerate}
\end{proposition}

We compute the point varieties $E_B$ of algebras $B$ mentioned in Theorem \ref{thm-ske4} (2). 
\begin{itemize}
\item[Case 1.] $B = \Bbbk[x_1, x_2, x_3, x_4]/(x_1^2 + x_2^2 + x_3^2 + x_4^2)$. Since $B$ is commutative, by \cite[Lemma 2.5]{HMM23}, we have
$E_B = \mathcal{Z}(x_1^2+x_2^2+x_3^2+x_4^2)\subset \mathbb{P}^3$. 
\item[Case 2.] $B = A/(x_1^2 + x_2^2 + x_3^2 + x_4^2)$ where $A$ is defined by $\mathcal{Q}$ as in (\ref{equ-case2}).
Then
$
(f_1, f_2, f_3, f_4, f_5, f_6, f) = (x_1, x_2, x_3, x_4) M 
$
where 
$$
M = \begin{pmatrix} 
x_2 & x_3 & -x_4 & 0 & 0 & 0 & x_1 \\ 
x_1 & 0 & 0 & x_3 & x_4 & 0 & x_2 \\
0 & x_1 & 0 & x_2 & 0 & x_4 & x_3 \\
0 & 0 & x_1 & 0 & x_2 & x_3 & x_4
\end{pmatrix}.
$$
The nonzero distinct (up to scalar) $4$-minors of $M$ are
\begin{gather*}
2 x_1^2 x_2 x_3, -2 x_1 x_2^2 x_3, -2 x_1 x_2 x_3^2, -2 x_2^2 x_3 x_4,  -2 x_2 x_3^2 x_4, 2 x_2 x_3 x_4^2, 2 x_1 x_2 x_3 x_4, \\
-x_1^4 + x_1^2 x_2^2 + x_1^2 x_3^2 - x_1^2 x_4^2, -x_1^2 x_2^2 + x_2^4 - x_2^2 x_3^2 - x_2^2 x_4^2,   \\
x_1^2 x_3^2 + x_2^2 x_3^2 - x_3^4 + x_3^2 x_4^2, -x_1^2 x_4^2 + x_2^2 x_4^2 + x_3^2 x_4^2 - x_4^4, \\
x_1^3 x_4 - x_1 x_2^2 x_4 - x_1 x_3^2 x_4 + x_1 x_4^3, \\
-x_1^3 x_2 + x_1 x_2^3 + x_1 x_2 x_3^2 - x_1 x_2 x_4^2, x_1^3 x_2 - x_1 x_2^3 + x_1 x_2 x_3^2 + x_1 x_2 x_4^2, \\
-x_1^3 x_3 + x_1 x_2^2 x_3 + x_1 x_3^3 - x_1 x_3 x_4^2,  -x_1^3 x_3 - x_1 x_2^2 x_3 + x_1 x_3^3 - x_1 x_3 x_4^2, \\
\end{gather*}
\begin{gather*}
-x_1^2 x_2 x_3 + x_2^3 x_3 - x_2 x_3^3 + x_2 x_3 x_4^2, x_1^2 x_2 x_3 + x_2^3 x_3 - x_2 x_3^3 - x_2 x_3 x_4^2,\\
x_1^2 x_2 x_4 - x_2^3 x_4 - x_2 x_3^2 x_4 + x_2 x_4^3, x_1^2 x_2 x_4 - x_2^3 x_4 + x_2 x_3^2 x_4 + x_2 x_4^3, \\ 
x_1^2 x_3 x_4 - x_2^2 x_3 x_4 - x_3^3 x_4 + x_3 x_4^3, -x_1^2 x_3 x_4 - x_2^2 x_3 x_4 + x_3^3 x_4 - x_3 x_4^3. 
\end{gather*}
By direct calculation and Lemma \ref{lem-E1}, we have
$$
E_B = X_M = \{(0\mathpunct{:} 1\mathpunct{:} \pm 1\mathpunct{:} 0)\} \cup \mathcal{Z}(F_1)\cup \mathcal{Z}(F_2) \subset \mathbb{P}^3
$$ 
where $F_1 = \{x_2, -x_1^2 + x_3^2 - x_4^2\}$, $F_2 =  \{x_3, -x_1^2 + x_2^2 - x_4^2\}$. 
\item[Case 3.] $B = A/(x_1^2 + x_2^2 + x_3^2 + x_4^2)$ where $A$ is defined by
$
\mathcal{Q} = (q_{ij}) \in \mathbb{M}_4(\Bbbk)
$
with $q_{ij} = -1$ for all $i\neq j$. By a direct calculation similar to Case 2, we have
\[
E_B = \left\{ 
\begin{aligned}
	&(1\mathpunct{:}0 \mathpunct{:}0 \mathpunct{:} \pm 1), \ (0 \mathpunct{:} 1\mathpunct{:} 0 \mathpunct{:} \pm 1), \ (0 \mathpunct{:} 0 \mathpunct{:} 1 \mathpunct{:} \pm 1), \\
	&(1\mathpunct{:}0\mathpunct{:}  \pm 1 \mathpunct{:} 0), \ (0 \mathpunct{:}1\mathpunct{:} \pm 1\mathpunct{:} 0), \ (1 \mathpunct{:}\pm 1\mathpunct{:}0\mathpunct{:}0)
\end{aligned}
\right\}.
\]
\end{itemize}

\begin{corollary}
For smooth $\pm 1$-skew quadric surfaces $B = A/(x_1^2 + x_2^2 + x_3^2 +x_4^2)$, $B' = A'/(x_1^2 + x_2^2 + x_3^2 +x_4^2)$, the following are equivalent:
\begin{enumerate}
\item $ \GrM B \cong \GrM B' $.
\item $ \qgr B \cong \qgr B' $.
\item $ \mathrm{D}^{\mathrm{b}}(\qgr B) \cong \mathrm{D}^{\mathrm{b}}(\qgr B') $.
\item $E_B \cong E_{B'}$.
\end{enumerate}
\end{corollary}

\end{document}